\documentclass[12pt]{amsart}
\usepackage{amsfonts,amssymb,amscd,amsmath,enumerate,verbatim,color}
\usepackage[latin1]{inputenc}
\usepackage{amscd}
\usepackage{latexsym}

\usepackage[dvipdfmx]{graphicx}
\usepackage{mathptmx}

\def\NZQ{\mathbb}               

\def\ZZ{{\NZQ Z}}
\def\RR{{\NZQ R}}

\def\PP{{\NZQ P}}

\def\ab{{\mathbf a}}

\def\eb{{\mathbf e}}

\def\vb{{\mathbf v}}
\def\ub{{\mathbf u}}

\def\opn#1#2{\def#1{\operatorname{#2}}} 
\opn\gr{gr}

\def\Oc{{\mathcal O}}
\def\Pc{{\mathcal P}}
\def\Qc{{\mathcal Q}}

\def\Vol{{\textnormal{Vol}}}

\newtheorem{Theorem}{Theorem}[section]
\newtheorem{Lemma}[Theorem]{Lemma}
\newtheorem{Corollary}[Theorem]{Corollary}
\newtheorem{Proposition}[Theorem]{Proposition}

\theoremstyle{definition}
\newtheorem{Remark}[Theorem]{Remark}

\newtheorem{Example}[Theorem]{Example}

\newtheorem{Conjecture}[Theorem]{Conjecture}

\let\epsilon\varepsilon
\let\phi=\varphi
\let\kappa=\varkappa
\opn\dis{dis}
\opn\height{height}
\opn\dist{dist}
\def\pnt{{\raise0.5mm\hbox{\large\bf.}}}

\opn\Lex{Lex}
\opn\conv{conv}
\opn\codeg{codeg}
\opn\codim{codim}
\opn\int{int}

\begin{document}

\title{Castelnuovo polytopes}
\author{Akiyoshi Tsuchiya}

\address{Akiyoshi Tsuchiya,
Graduate School of Mathematical Sciences,
University of Tokyo,
Komaba, Meguro-ku, Tokyo 153-8914, Japan} 
\email{akiyoshi@ms.u-tokyo.ac.jp}


\begin{abstract}
It is known that the sectional genus of a polarized variety has an upper bound, which is an extension of the Castelnuovo bound on the genus of a projective curve.
Polarized varieties whose sectional genus achieves this bound are called Castelnuovo.
On the other hand, a lattice polytope is called Castelnuovo if the associated polarized toric variety is Castelnuovo. Kawaguchi characterized Castelnuovo polytopes having interior lattice points in terms of their $h^*$-vectors. In this paper, as a generalization of this result, a characterization of all Castelnuovo polytopes will be presented.
Finally, as an application of our characterization, we give a sufficient criterion for a lattice polytope to be IDP.
\end{abstract}

\maketitle

\section{Introduction}
For an $n$-dimensional complex projective variety $X$ and an ample line bundle $L$ on $X$, the pair $(X,L)$ is called an $n$-dimensional \textit{polarized variety}. 
Let $\chi(tL)$ be the \textit{Euler-Poincar\'{e} characteristic} of $tL$. Then $\chi(tL)$ is a polynomial in $t$ of degree $n$.
We put
\[
\chi(tL)=\sum_{j=0}^n \chi_j(X,L) \dfrac{t^{[j]}}{j!},
\]
where $t^{[j]}=t(t+1)\cdots (t+j-1)$ for $j \geq 1$ and $t^{[0]}=1$.
Then the \textit{sectional genus} of $(X,L)$, denoted by $g(X,L)$, is defined by 
\[
g(X,L)=1-\chi_{n-1}(X,L).
\] 
The sectional genus $g(X,L)$ plays an important role in the classification theory of polarized varieties (cf. \cite{Fujita}).
Fujita gave an upper bound for the sectional genus as a higher-dimensional version of the Castelnuovo bound.
\begin{Theorem}[{\cite[Theorem 16.3]{Fujita}}]
\label{thm:castel}
	Let $X$ be an $n$-dimensional complex projective variety and let $L$ be a line bundle on $X$. Assume that $h^0(L) \geq n+2$, $L$ is basepoint free, and the morphism defined by $L$ is birational on its image. Then one has
	\[
	g(X,L) \leq m \Delta(X,L) - \dfrac{1}{2}m(m-1)(L^n-\Delta(X,L)-1),
	\]
	where $m=\lfloor(L^n-1)/(L^n-\Delta(X,L)-1) \rfloor$ and $\Delta(X,L)=L^n+n-h^0(L)$, which is called the {\rm $\Delta$-genus} of $(X,L)$.
\end{Theorem}
We remark that the birationality of the morphism defined by $L$ in Theorem \ref{thm:castel} is essential (see \cite[Remark 16.5]{Fujita}).
A polarized variety $(X,L)$ is called \textit{Castelnuovo} if it satisfies the assumption of Theorem $\ref{thm:castel}$ and the sectional genus achieves the upper bound in Theorem \ref{thm:castel}.
Even if $(X,L)=(\PP^n,\Oc(1))$, which does not satisfy the condition $h^0(L) \geq n+2$, we say that $(X,L)$ is Castelnuovo. 
Castelnuovo varieties contain many important examples such as the higher-dimensional version of del Pezzo surfaces and $K3$ surfaces. 
We refer the reader to \cite{Fujita} for the detailed information about the sectional genus and Castelnuovo varieties.
In the present paper, we focus on Castelnuovo toric varieties.

A \textit{lattice polytope} is a convex polytope all of whose vertices have integer coordinates. A full-dimensional lattice polytope defines a polarized toric variety and there is a one-to-one correspondence between full-dimensional lattice polytopes and polarized toric varieties.
In particular, the ample line bundle associated to a full-dimensional lattice polytope is basepoint free.
We can read off properties of polarized toric varieties from the associated lattice polytopes and vice versa.
A lattice polytope is called \textit{Castelnuovo} if the associated polarized toric variety is Castelnuovo.
In \cite{Kawaguchi}, Kawaguchi characterized Castelnuovo polytopes having interior lattice points in terms of an invariant of lattice polytopes.
Let $\Pc \subset \RR^n$ be a full-dimensional lattice polytope
and define $h^*(\Pc,t)$ by the formula
\[
h^*(\Pc,t)=(1-t)^{n+1} \left[1+\sum_{k=1}^{\infty}|k\Pc \cap \ZZ^n|t^k\right],
\]
where $k\Pc=\{k \ab : \ab \in \Pc\}$, the dilated polytopes of $\Pc$.
Then it is known that $h^*(\Pc,t)$ is a polynomial in $t$ of degree at most $n$ with nonnegative integer coefficients (\cite{Stanleynonnegative}) and it
is called
the \textit{$h^*$-polynomial} of $\Pc$. 
Letting $h^*(\Pc,t)=\sum_{i=0}^{n} h_i^* t^i$, the sequence $h^*(\Pc)=(h^*_0,\ldots,h^*_n)$ is called the \textit{$h^*$-vector} of $\Pc$. It then follows that $h^*_0=1$, $h^*_1=|\Pc \cap \ZZ^n|-(n +1)$ and $h^*_{n}=|{\rm int} (\Pc) \cap \ZZ^n|$, where ${\rm int} (\Pc)$ is the  interior of $\Pc$.
Furthermore, $\sum_{i=0}^{n} h_i^*$ is equal to the \textit{normalized volume} of $\Pc$, denoted by $\Vol(\Pc)$.
We refer the reader to \cite{BeckRobins} for the detailed information about  $h^*$-polynomials and $h^*$-vectors.

For a lattice polytope having interior lattice points, a lower bound on its $h^*$-vector  is known.
\begin{Theorem}[{\cite[Hibi's Lower Bound Theorem]{Hibi}}]
\label{thm:hibi}
	Let $\Pc \subset \RR^n$ be a full-dimensional lattice polytope with ${\rm int}(\Pc) \cap \ZZ^n \neq \emptyset$.
	Then we have  $h^*_1 \leq h^*_j$ for any $2 \leq j \leq n-1$.
\end{Theorem}
We can compute the sectional genus of a polarized toric variety by the $h^*$-vector of the associated lattice polytope (see Section 2).
In \cite{Kawaguchi}, Kawaguchi proved that a lattice polytope having interior lattice points is Castelnuovo if and only if its $h^*$-vector achieves the lower bound in Theorem \ref{thm:hibi}. 
\begin{Theorem}[{\cite[Theorem 1.3]{Kawaguchi}}]
	\label{thm:Kawaguchi}
	Let $\Pc \subset \RR^n$ be a full-dimensional lattice polytope with ${\rm int}(\Pc) \cap \ZZ^n \neq \emptyset$.
	Then $\Pc$ is Castelnuovo if and only if  $h^*_1 = h^*_j$ for any $2 \leq j \leq n-1$. 
\end{Theorem}
This result is equivalent to a characterization of Castelnuovo toric varieties $(X,L)$ with $h^0(L+K_X) \geq 1$.
On the other hand, with respect to a polarized toric variety with $h^0(L+K_X) = 0$, a general criterion to determine whether or not it is Castelnuovo, namely, a characterization of Castelnuovo polytopes without interior lattice points is not known.

In the present paper, as a generalization of Theorem \ref{thm:Kawaguchi} we give a characterization of all Castelnuovo polytopes. Denote $\deg(\Pc)$ the degree of the $h^*$-polynomial of $\Pc$. It then follows from \cite{Baty} that 
\[\deg(\Pc)=n+1-\min\{k \in \ZZ_{\geq 1}  : {\rm int}(k\Pc) \cap \ZZ^n \neq \emptyset\}.\]
A lattice polytope $\Pc$ is called \textit{spanning} if every lattice point in $\ZZ^n$ is affine integer combination of the lattice points in $\Pc$. Note that $\Pc$ is spanning if and only if the morphism defined by the ample line bundle associated to $\Pc$ is birational on its image (\cite[Proposition 2.11]{HKN2}).
Hence Castelnuovo polytopes are spanning.
Recently in \cite{HKN2} a generalization of Hibi's lower bound theorem was proven. When we restrict spanning lattice polytopes, one has the following.

\begin{Theorem}[{\cite[Corollary 1.6]{HKN2}}]
\label{thm:general}
	Let $\Pc \subset \RR^n$ be a full-dimensional spanning lattice polytope. Then we have $h^*_1 \leq  h^*_j$ for any $2 \leq j \leq \deg(\Pc)-1$.
\end{Theorem}
We will show that a lattice polytope is Castelnuovo if and only if it is spanning and its $h^*$-vector achieves the lower bound in Theorem \ref{thm:general} and an additional condition is satisfied. 
In fact, the following is the main theorem of the present paper.
\begin{Theorem}
\label{main}
	Let $\Pc \subset \RR^n$ be a full-dimensional lattice polytope.
	Then $\Pc$ is Castelnuovo if and only if $\Pc$ is spanning, $h^*_1 \geq h^*_{\deg(\Pc)}$ and $h^*_1 = h^*_j$ for any $2 \leq j \leq \deg(\Pc)-1$.
\end{Theorem}
The proof of Theorem \ref{thm:Kawaguchi} in \cite{Kawaguchi} missed the birationality of the morphism defined by the ample line bundle associated to a full-dimensional lattice polytope $\Pc \subset \RR^n$ with ${\rm int}(\Pc) \cap \ZZ^n \neq \emptyset$ and $h^*_1=h^*_j$ for any $2 \leq j \leq n-1$. 
In the present paper, we will fill this gap.  
Moreover, we give an example which shows that we need the spanningness condition in Theorem \ref{main} (Example \ref{ex:nonspanning}).
Finally, in Section 4, we give a new sufficient criterion for a lattice polytope to be IDP (Theorem \ref{thm:application}) as an application of Theorem \ref{main}.

\section{sectional genus}
In this section, we recall formulas for the sectional genus of a polarized toric variety and the upper bound in Theorem \ref{thm:castel} in terms of their $h^*$-vectors of the associated lattice polytope.
Let $(X,L)$ be an $n$-dimensional polarized toric variety and $\Pc \subset \RR^n$ the associated lattice polytope.
Set $s=\deg(\Pc)$.
Then we can read off many invariants of $L$ from $\Pc$ (see cf. \cite{Oda}). 
In particular, one has
\begin{itemize}
	\item $h^0(L)=|\Pc \cap \ZZ^n|=h^*_1+(n+1)$,
	\item $L^n=\Vol(\Pc)=\sum_{j=0}^{s}h^*_j$.
\end{itemize}
Hence we obtain $\Delta(X,L)=h^*_2+\cdots+h^*_s$.
Moreover, letting $m$ be the integer defined in Theorem \ref{thm:castel}, one has
\[
m=\left\lfloor \dfrac{h^*_1+\cdots+h^*_{s}}{h^*_1}\right\rfloor. 
\]
Hence we obtain 
\[
m \Delta(X,L) - \dfrac{m(m-1)}{2}(L^n-\Delta(X,L)-1)=m(h^*_2+\cdots+h^*_s)-\dfrac{m(m-1)}{2}h^*_1.
\]
On the other hand, the sectional genus $g(X,L)$ can be also expressed in terms of the $h^*$-vector.
\begin{Lemma}[{\cite[p. 6]{Kawaguchi}}]
	Let $(X,L)$ be an $n$-dimensional polarized toric variety and $\Pc$ the associated lattice polytope.
	Then one has
	\[
	g(X,L)=\sum_{j=1}^{\deg(\Pc)}(j-1)h^*_j.
	\]
\end{Lemma}
Therefore, we can determine whether or not the sectional genus of the associated polarized toric variety achieves the upper bound in Theorem \ref{thm:castel} from the $h^*$-vector of the polytope.

\section{Proof of Theorem \ref{main}}
In this section, we prove Theorem \ref{main}.
Let ${\bf 0}$ be the origin of $\RR^n$ and $\eb_i$ the $i$-th unit coordinate vector in $\RR^n$. The \textit{standard simplex} of dimension $n$ is the convex hull of ${\bf 0}, \eb_1,\ldots, \eb_n$. It is well known that the associated polarized toric variety of a full-dimensional lattice polytope of dimension $n$ is $(\PP^n, \Oc(1))$ if and only if the polytope is unimodularly equivalent to the standard simplex of dimension $n$. Here two lattice polytopes $\Pc, \Qc \subset \RR^n$ are said to be \textit{unimodularly equivalent} if there exists $f \in GL_n(\ZZ)$ and $\ub \in \ZZ^n$ such that $\Qc=f(\Pc)+\ub$.

Next we recall a geometric interpretation for spanning polytopes.
 \begin{Lemma}[{\cite[Proposition 2.11]{HKN2}}]
 \label{lem:spanning}
 A full-dimensional lattice polytope is spanning if and only if the morphism defined by the associated ample line bundle is birational on its image. 
 \end{Lemma} 

Now, we prove a sufficient criterion for a lattice polytope to be Castelnuovo.
\begin{Proposition}
\label{thm:nec}
	Let $\Pc \subset \RR^n$ be a full-dimensional lattice polytope. If $\Pc$ is spanning, $h^*_1 \geq h^*_{\deg(\Pc)}$ and $h^*_1 = h^*_j$ for any $2 \leq j \leq \deg(\Pc)-1$, then $\Pc$ is Castelnuovo. 
\end{Proposition}
\begin{proof}
Assume that $\Pc$ is not unimodularly equivalent to the standard simplex of dimension $n$.
	Let $(X,L)$ be the associated polarized toric variety of $\Pc$ and $m$ the integer defined in Theorem \ref{thm:castel}.
	Since $\Pc$ is spanning, the morphism defined by $L$ is birational on its image by Lemma \ref{lem:spanning}.
	Set $s=\deg(\Pc)$.
	Then one has
	\[
	L^n=1+(s-1)h^*_1+h^*_s.
	\]
	Hence we obtain
	\[
	m= s-1 + \left\lfloor \dfrac{h^*_{s}}{h^*_1} \right\rfloor = \begin{cases}
		s & (h^*_s = h^*_1),\\
		s-1 & (h^*_s < h^*_1).
	\end{cases}
	\]
	
	Now, assume that $m=s$, hence $h^*_s=h^*_1$.
	Then one has 
	\begin{align*}
	m \Delta(X,L) - \dfrac{1}{2}m(m-1)(L^n-\Delta(X,L)-1)
	=&\dfrac{1}{2}s(s-1)h^*_1\\
	=&\sum_{j=1}^s(j-1)h^*_1\\
	=&\sum_{j=1}^s(j-1)h^*_j=g(X,L).
	\end{align*}
	This implies that $g(X,L)$ achieves the upper bound in Theorem \ref{thm:castel}. Therefore, $(X,L)$ is Castelnuovo. Hence $\Pc$ is Castelnuovo. 
	
	Next, we assume that $m=s-1$, hence $h^*_s < h^*_1$.
	Then similarly, one has 
		\begin{align*}m \Delta(X,L) - \dfrac{1}{2}m(m-1)(L^n-\Delta(X,L)-1)
	=&\dfrac{1}{2}(s-1)(s-2)h^*_1+(s-1)h^*_s\\
	=&\sum_{j=1}^{s-1}(j-1)h^*_1+(s-1)h^*_s\\
	=&\sum_{j=1}^s(j-1)h^*_j=g(X,L).
	\end{align*}
		This implies that $g(X,L)$ achieves the upper bound in Theorem \ref{thm:castel}. Therefore, $(X,L)$ is Castelnuovo, namely, $\Pc$ is Castelnuovo, as desired.
	\end{proof}

	From the lower bound in Theorem \ref{thm:general} we obtain the following lower bound for the normalized volume of a spanning lattice polytope $\Pc$.
	\begin{Lemma}
	\label{lem:vol}
		Let $\Pc \subset \RR^n$ be a full-dimensional spanning lattice polytope.
		Then one has
		\[
		\Vol(\Pc) \geq 1+(\deg(\Pc)-1)h^*_1+h^*_{\deg(\Pc)}
		\]
		and the equality holds if and only if for any $2 \leq j \leq \deg(\Pc)-1$, $h^*_1=h^*_j$.
	\end{Lemma}
	
	Now, we prove a necessary criterion for a lattice polytope to be Castelnuovo.
	
	\begin{Proposition}
	\label{prop:suf}
	Let $\Pc \subset \RR^n$ be a full-dimensional Castelnuovo 
	 polytope.
	It then follows that $\Pc$ is spanning and 
	for any $2 \leq j \leq \deg(\Pc)-1$, $h^*_1 = h^*_j$ and $h^*_1 \geq h^*_{\deg(\Pc)}$.
	\end{Proposition}
	\begin{proof}
	Assume that $\Pc$ is not unimodularly equivalent to the standard simplex of dimension $n$.
		Let $(X,L)$ be the associated polarized toric variety of $\Pc$ and $m$ the integer defined in Theorem \ref{thm:castel}.
		Set $s=\deg(\Pc)$.
		Since $\Pc$ is Castelnuovo, the morphism defined by $L$ is birational on its image. Hence $\Pc$ is spanning by Lemma \ref{lem:spanning}.
		From Theorem \ref{thm:general} we obtain
			\[
			m \geq \left\lfloor \dfrac{(s-1)h^*_1+h^*_s}{h^*_1} \right\rfloor=(s-1)+\left\lfloor \dfrac{h^*_s}{h^*_1} \right\rfloor \geq s-1.
			\]
			
			First, assume that $m=s-1$ hence $h^*_1 > h^*_s$. 
			Since $\Pc$ is Castelnuovo, we obtain
			\[
			g(X,L)=(s-1)(\Vol(\Pc)-h^*_1-1)-\dfrac{(s-1)(s-2)}{2}h^*_1.	
			\]
			Since $\Pc$ is spanning, by Theorem \ref{thm:general}, one has
			\begin{align*}
				0 &\leq \sum_{j=1}^{s-1}(s-1-j)(h^*_j-h^*_1)\\
				&=\sum_{j=1}^s(s-1-j)h^*_j+h^*_s-\dfrac{(s-1)(s-2)}{2}h^*_1\\
				&=\sum_{j=1}^s(s-2)h^*_j+\sum_{j=1}^{s}(1-j)h^*_j+h^*_s-\dfrac{(s-1)(s-2)}{2}h^*_1\\
				&=(s-2)(\Vol(\Pc)-1)-g(X,L)+h^*_s-\dfrac{(s-1)(s-2)}{2}h^*_1\\
				&=-\Vol(\Pc)+1+(s-1)h^*_1+h^*_s \leq 0,
			\end{align*}
			where the last inequality follows from Lamma \ref{lem:vol}.
			Note that this is essentially the same argument which was used in \cite{Kawaguchi}.
			Hence we obtain $\Vol(\Pc)=1+(s-1)h^*_1+h^*_s$. From Lemma \ref{lem:vol} this implies $h^*_1 = h^*_j$ for any $2 \leq j \leq s-1$.
			
			Next, we assume that $m \geq s$, hence $h^*_1 \leq h^*_s$. Let $k \geq 1$ be the integer such that $kh^*_1 \leq h^*_s < (k+1)h^*_1$. Hence one has $m=s-1+k$.
			Since $\Pc$ is Castelnuovo and spanning, it then follows that
			\begin{align*}
				 0&=m\Delta(X,L)-\dfrac{m(m-1)}{2}(L^n-\Delta(X,L)-1)-g(X,L)\\
				&=\underbrace{(s-1+k)\sum_{j=2}^{s}h^*_j}-\dfrac{(s-1+k)(s-2+k)}{2}h^*_1\underbrace{-\sum_{j=1}^s(j-1)h^*_j}\\
				&=\left(\underbrace{\sum_{j=1}^{s}(s+k-j)h^*_j}-(s-1+k)h^*_1\right)-\dfrac{(s-1+k)(s-2+k)}{2}h^*_1\\
				&=\left(\underbrace{\sum_{j=1}^{s}(s+k-j)(h^*_j-h^*_1)}+\sum_{j=1}^{s}(s+k-j)h^*_1\right)-(s-1+k)h^*_1-\dfrac{(s-1+k)(s-2+k)}{2}h^*_1\\
				& \geq k(h^*_{s}-h^*_{1})\underbrace{+\sum_{j=1}^{s}(s+k-j)h^*_1-(s-1+k)h^*_1-\dfrac{(s-1+k)(s-2+k)}{2}h^*_1}\\
				& = k(h^*_s-h^*_1)-\dfrac{k(k-1)}{2}h^*_1\\
				&\geq k(k-1)h^*_1-\dfrac{k(k-1)}{2}h^*_1 = \dfrac{k(k-1)}{2}h^*_1\geq 0,
			\end{align*}
			where the first inequality follows from Theorem \ref{thm:general} and the second inequality follows from $kh^*_1 \leq h^*_s$.
			Therefore it holds that $k=1$, $h^*_1=h^*_s$ and $h^*_1 = h^*_j$ for any $2 \leq j \leq s-1$, as desired.
	\end{proof}
	Therefore, we can complete a proof of Theorem \ref{main} by combining  Propositions \ref{thm:nec} and \ref{prop:suf}.

Next, we complete a proof of Theorem \ref{thm:Kawaguchi}.
We say that a full-dimensional lattice polytope $\Pc \subset \RR^n$ possesses the \textit{integer decomposition property}  if for every integer $k \geq 1$, every lattice point in $k\Pc$ is a sum of $k$ lattice points from $\Pc$.
	A lattice polytope which possesses the integer decomposition property is called \textit{IDP}.
	In general, IDP polytopes are spanning.
On the other hand, a (lattice) triangulation of a full-dimensional lattice polytope is called \textit{unimodular} if every maximal face of the triangulation is unimodularly equivalent to the standard simplex. It then follows that a lattice polytope with a unimodular triangulation is IDP.
We can determine whether a triangulation of a lattice polytope is unimodular or not in terms of the $h^*$-vector of the polytope.
Let $\Delta$ be a  simplicial complex of dimension $n-1$ with $f_i$ $i$-dimensional faces and $f_{-1}=1$. Then the \textit{$h$-vector} $h(\Delta)=(h_0,\ldots,h_n)$ of $\Delta$ is defined by the relation
\[
\sum_{i=0}^{n} f_{i-1} (t-1)^{n-i} = \sum_{i=0}^{n} h_i t^{n-i}.
\]
\begin{Lemma}[{\cite[Theorem 2]{BM}}]
\label{lem:complex}
Let $\Pc$ be a lattice polytope of dimension $n$, and let $\Delta$ be a triangulation of $\Pc$ with $h(\Delta)=(h_0,\ldots,h_{n+1})$. Then $\Delta$ is unimodular if and only if $h^*(\Pc)=(h_0,\ldots,h_n)$.
\end{Lemma}

Now, we show the following. 
\begin{Lemma}
\label{lem:unimodular}
Let	$\Pc \subset \RR^n$ be a full-dimensional lattice polytope  with ${\rm int}(\Pc) \cap \ZZ^n \neq \emptyset$ and $h^*_1=h^*_j$ for any $2 \leq j \leq n-1$. Then $\Pc$ has a unimodular triangulation. In particular, $\Pc$ is IDP and spanning.
\end{Lemma}

\begin{proof}
	Set ${\rm int}(\Pc)\cap \ZZ^n=\{\vb_1,\ldots,\vb_{l}\}$. 
	We take any triangulation $\Delta(0)$ of the boundary $\partial \Pc$ with the vertex set $\partial \Pc \cap \ZZ^n$.
	Let $\Delta(j)$ be the triangulation of $\Pc$ with the vertex set $(\partial\Pc \cap \ZZ^n) \cup \{\vb_1,\ldots,\vb_j\}$ for each $1 \leq j \leq l$ defined in the proof of \cite[Theorem 1.1]{Hibi}.
	We let $(h_0,\ldots,h_n,h_{n+1})$ be the $h$-vector of $\Delta(l)$.
	Then one has $h_0=h^*_0=1$, $h_1=h^*_1$, $h_n=h^*_n$ and $h_{n+1}=0$.
	Moreover, it follows from the proof of \cite[Theorem 1.1]{Hibi} that for any $2 \leq j \leq n-1$, we obtain
	$h^*_1 \leq h_j \leq h^*_j$.
	Since for any $2 \leq j \leq n-1$, $h^*_1=h^*_j$, one has  $h^*_j=h_j$. Hence, we have $h^*(\Pc)=(h_0,\ldots,h_{n})$.
	It then from Lemma \ref{lem:complex} that $\Delta(l)$ is a unimodular triangulation of $\Pc$, as desired.
\end{proof}
Therefore, by combining Theorem \ref{main} and Lemma \ref{lem:unimodular}, we can complete a proof of Theorem \ref{thm:Kawaguchi}.

\begin{proof}[Proof of Theorem \ref{thm:Kawaguchi}]
Since $\Pc$ has an interior lattice point, we obtain ${\rm deg}(\Pc)=n$.
	From Lemma \ref{lem:unimodular} if for any $2 \leq j \leq n-1$, $h^*_1 = h^*_j$, then $\Pc$ is spanning. On the other hand, it always satisfies $h^*_1 \geq h^*_n$.
	Therefore, from Theorem \ref{main} we know that $\Pc$ is Castelnuovo if and only if for any $2 \leq j \leq n-1$, $h^*_1 = h^*_j$. 
\end{proof}

Finally, we give an example which says that we need the spanningness condition in Theorem \ref{main}.
Namely, we can not determine whether or not a lattice polytope is Castelnuovo by using only the $h^*$-vector.

\begin{Example}
\label{ex:nonspanning}
	Let $\Pc \subset \RR^4$ be the lattice polytope which is the convex hull of
	\[
	{\bf 0}, \eb_1,\eb_2,\eb_3, \eb_1+\eb_2+2\eb_4, \eb_1-\eb_3 \subset \RR^4.
	\]
	Then one has $h^*(\Pc)=(1,1,1,1,0)$.
	On the other hand, it is clear that $\Pc$ is not spanning.  
	In particular, $\Pc$ is not Castelnuovo.
\end{Example}

	\section{An application of the main theorem}
	In this section, we give an application of Theorem \ref{main}.
	Let $\Pc$ be a full-dimensional lattice polytope and $(X,L)$ the associated polarized toric variety of $\Pc$. 
	We say that $\Pc$ is \textit{very ample} if for any sufficiently large $k \in \ZZ$ every lattice point in $k\Pc$ is a sum of $k$ lattice points in $\Pc$. 
It then follows that IDP lattice polytopes are very ample, and very ample lattice polytopes are spanning.
	Note that $\Pc$ is very ample if and only if the associated ample line bundle is very ample (cf. \cite[Section 6]{toricvarieties}).
In particular, if $\Pc$ is very ample, then $L$ is normally generated if and only if $\Pc$ is IDP (cf. \cite[Theorem 5.4.8]{toricvarieties}).
	In \cite[p. 141]{Fujita}, it is shown that for Castelnuovo varieties $(X,L)$, $L$ is very ample and normally generated. Therefore, Castelnuovo polytopes are IDP.
	Thus from Theorem \ref{main} we obtain the following.
	\begin{Theorem}
	\label{thm:application}
		Let $\Pc$ be a full-dimensional spanning lattice polytope.
		If for any $2 \leq j \leq \deg(\Pc)-1$, $h^*_1 = h^*_j$ and $h^*_1 \geq h^*_{\deg(\Pc)}$, then $\Pc$ is IDP.
	\end{Theorem}
	
If ${\rm deg}(\Pc)=2$, then we do not need the spanningness assumption in Theorem \ref{thm:application}.
In fact, one has the following.
\begin{Corollary}[{\cite[Corollary 1.2]{KY}}]
Let $\Pc \subset \RR^n$ be a full-dimensional lattice polytope with $\deg(\Pc)=2$.
If $h^*_1 \geq h^*_2$, then $\Pc$ is IDP.
\end{Corollary}	
\begin{proof}
	In \cite[Corollary 3.5]{KY}, it is shown that $\Pc$ is spanning. 
	Hence $\Pc$ is IDP from Theorem \ref{thm:application}, as desired.
\end{proof}

\begin{Remark}
	If ${\rm deg}(\Pc) \geq 3$, then we need the spanningness assumption in Theorem \ref{thm:application} (see Example \ref{ex:nonspanning}). 
\end{Remark}	
	
A full-dimensional lattice polytope $\Pc \subset \RR^n$ is called \textit{smooth} if it is simple and if its primitive edge directions at every vertex form a basis of $\ZZ^n$. Smooth lattice polytopes correspond to smooth polarized toric varieties. It is well-known that the ample line bundle associated to a smooth lattice polytope is very ample.
In \cite{Odaconj}, Oda conjectured the following.
\begin{Conjecture}[Oda's Conjecture]
\label{conj:oda}
Every smooth lattice polytope is IDP.
\end{Conjecture}
Since Castelnuovo polytopes are IDP, Conjecture \ref{conj:oda} holds for smooth Castelnuovo polytopes. In particular, we obtain the following.
\begin{Corollary}
	Let $\Pc$ be a full-dimensional smooth lattice polytope.
		If for any $2 \leq j \leq \deg(\Pc)-1$, $h^*_1 = h^*_j$ and $h^*_1 \geq h^*_{\deg(\Pc)}$, then $\Pc$ is IDP.
\end{Corollary}
	
Finally, we give an example of a spanning lattice polytope which is not IDP and whose $h^*$-vector achieves the lower bound in Theorem \ref{thm:general}.
This implies that the condition $h^*_1 \geq h^*_{{\rm deg}(\Pc)}$ in Theorem \ref{thm:application} is necessary.

	\begin{Example}
		Let $\Pc \subset \RR^{2a+1}$ be the lattice polytope which is the convex hull of
		\[
		{\bf 0}, \eb_1, \ldots, \eb_{2a}, \sum_{i=1}^{a}\eb_i+\sum_{j=a+1}^{2a}2\eb_i+3\eb_{2a+1},-\sum_{i=a+2}^{2a+1}\eb_i.
		\]
		Then $\Pc$ is spanning and one has		\[
		(h^*_0,\ldots,h^*_{2a+1})=(1,\underbrace{1,\ldots,1}_a,2,\underbrace{0,\ldots,0}_a).
		\]
		Hence we obtain $\deg(\Pc)=a+1$ and for any $2 \leq j \leq \deg(\Pc)-1$, $h^*_1=h^*_j$ and $h^*_1 < h^*_{\deg(\Pc)}$.
		In particular, \[
		\Pc \cap \ZZ^{2a+1}=\left\{{\bf 0}, \eb_1, \ldots, \eb_{2a}, \sum_{i=1}^{a}\eb_i+\sum_{j=a+1}^{2a}2\eb_i+3\eb_{2a+1},-\sum_{i=a+2}^{2a+1}\eb_i\right\}.\]
		On the other hand, one has \[
		\eb_1+\cdots+\eb_{2a+1} \in (a+1)\Pc \cap \ZZ^{2a+1}\] However, this lattice point can not be a sum of $a+1$ lattice points in $\Pc^{2a+1}$. Hence $\Pc$ is not IDP.
	\end{Example}

	\section*{Acknowledgements}
	I am also grateful to Professor Kohji Yanagawa for helpful comments on spanning polytopes. His comments improved Theorems \ref{main} and \ref{thm:application}.
	I am also grateful to Makoto Enokizono for fruitful discussions on Castelnuovo varieties. 
	The author would like to thank anonymous referees for reading the manuscript carefully.
	The author was partially supported by JSPS KAKENHI 19J00312 and 19K14505.


\begin{thebibliography}{99}
	\bibitem{Baty}
	V. V. Batyrev, Lattice polytopes with a given $h^*$-polynomial, \textit{Contemp. Math.}  \textbf{423} (2007), 1--10.
	

	\bibitem{BeckRobins}
M. Beck and S. Robins, ``Computing the continuous discretely'', Undergraduate Texts in Mathematics, Springer, second edition, 2015.

\bibitem{BM}
U. Betke and P. McMullen, Lattice points in lattice polytopes, \textit{Monatsh. Math.} \textbf{99} (1985), 253--265.

	
	\bibitem{toricvarieties}
	D. Cox, J. Little and H. Schenck, Toric Varieties, Graduate Texts in Mathematics, vol. {\bf 124}, American Mathematical Society (2011).
	
	\bibitem{Fujita}
	T. Fujita, Classification theories of polarized varieties, London Mathematical Society Lecture Note Series, vol. {\bf 155}, Cambridge University Press (1990).
	
	
	\bibitem{HKN2}
	J. Hofscheier, L. Katth\"{a}n and B. Nill,
	Spanning lattice polytopes and the uniform position principle,
	arXiv:1711.09512.
	
	
	\bibitem{hibi}
T. Hibi, Dual polytopes of rational convex polytopes, 
\textit{Combinatorica} {\bf 12} (1992), 237--240.
	
	\bibitem{Hibi}
	T. Hibi, A lower bound theorem for Ehrhart polynomials of convex polytopes,
	\textit{Adv. Math.} {\bf 105} (1994), 162--165.
	
	\bibitem{KY}
	L. Katth\"{a}n and K. Yanagawa,
	Graded Cohen--Macaulay domains and lattice polytopes with short $h$-vector,
	arXiv:1907.07214.

\bibitem{Kawaguchi}
R. Kawaguchi, Sectional genus and the volume of a lattice polytope, \textit{J. Algebraic Combin.} \textbf{53} (2021), 1253--1264.

\bibitem{Oda} T. Oda, Convex Bodies and Algebraic Geometry, Springer, Berlin (1988).

	\bibitem{Odaconj}
	T. Oda,
	Problems on Minkowski sums of convex lattice polytopes, arXiv:0812.1418.
	
\bibitem{Stanleynonnegative}
R.~P. Stanley,
Decompositions of rational convex polytopes,
\textit{Ann. Discrete Math.} {\bf 6} (1980), 333--342. 




\end{thebibliography}
\end{document}